\documentclass[12pt]{amsart}
\usepackage{fancyhdr}

\usepackage{verbatim}
\usepackage{amssymb, amsmath}
\usepackage{graphicx}
\usepackage[format=plain]{caption} 
\usepackage{subcaption}
\usepackage{xcolor}

\usepackage{url}

\usepackage{mathrsfs}
\usepackage{bibentry}
\usepackage[norefs,nocites]{refcheck}
 
\begin{document}
\newtheorem{thm}{Theorem}[section]
\newtheorem{lem}[thm]{Lemma}
\newtheorem{prop}[thm]{Proposition}
\newtheorem{cor}[thm]{Corollary}
\newtheorem{conj}[thm]{Conjecture}
\newtheorem{proj}[thm]{Project}
\newtheorem*{remark}{Remark}



\makeatletter
\@namedef{subjclassname@2020}{\textup{2020} Mathematics Subject Classification}
\makeatother

\newcommand{\rad}{\operatorname{rad}}

\newcommand{\Z}{{\mathbb Z}} 
\newcommand{\Q}{{\mathbb Q}}
\newcommand{\R}{{\mathbb R}}
\newcommand{\C}{{\mathbb C}}
\newcommand{\N}{{\mathbb N}}
\newcommand{\FF}{{\mathbb F}}
\newcommand{\fq}{\mathbb{F}_q}
\newcommand{\rmk}[1]{\footnote{{\bf Comment:} #1}}

\renewcommand{\mod}{\;\operatorname{mod}}
\newcommand{\ord}{\operatorname{ord}}
\newcommand{\TT}{\mathbb{T}}
\renewcommand{\i}{{\mathrm{i}}}
\renewcommand{\d}{{\mathrm{d}}}
\renewcommand{\^}{\widehat}
\newcommand{\HH}{\mathbb H}
\newcommand{\Vol}{\operatorname{vol}}
\newcommand{\area}{\operatorname{area}}
\newcommand{\tr}{\operatorname{tr}}
\newcommand{\norm}{\mathcal N} 
\newcommand{\intinf}{\int_{-\infty}^\infty}
\newcommand{\ave}[1]{\left\langle#1\right\rangle} 
\newcommand{\Var}{\operatorname{Var}}
\newcommand{\Prob}{\operatorname{Prob}}
\newcommand{\sym}{\operatorname{Sym}}
\newcommand{\disc}{\operatorname{disc}}
\newcommand{\CA}{{\mathcal C}_A}
\newcommand{\cond}{\operatorname{cond}} 
\newcommand{\lcm}{\operatorname{lcm}}
\newcommand{\Kl}{\operatorname{Kl}} 
\newcommand{\leg}[2]{\left( \frac{#1}{#2} \right)}  
\newcommand{\Li}{\operatorname{Li}}

\newcommand{\sumstar}{\sideset \and^{*} \to \sum}

\newcommand{\LL}{\mathcal L} 
\newcommand{\sumf}{\sum^\flat}
\newcommand{\Hgev}{\mathcal H_{2g+2,q}}
\newcommand{\USp}{\operatorname{USp}}
\newcommand{\conv}{*}
\newcommand{\dist} {\operatorname{dist}}
\newcommand{\CF}{c_0} 
\newcommand{\kerp}{\mathcal K}

\newcommand{\Cov}{\operatorname{cov}}
\newcommand{\Sym}{\operatorname{Sym}}

\newcommand{\Ht}{\operatorname{Ht}}

\newcommand{\E}{\operatorname{\mathbb E}} 
\newcommand{\sign}{\operatorname{sign}} 
\newcommand{\meas}{\operatorname{meas}} 

\newcommand{\divid}{d} 

\newcommand{\GL}{\operatorname{GL}}
\newcommand{\SL}{\operatorname{SL}}
\newcommand{\re}{\operatorname{Re}}
\newcommand{\im}{\operatorname{Im}}
\newcommand{\res}{\operatorname{Res}}

 \newcommand{\EWp}{\mathbb E^{\rm WP}_g} 
\newcommand{\orb}{\operatorname{Orb}}
\newcommand{\supp}{\operatorname{Supp}}
\newcommand{\mmfactor }{\textcolor{red}{c_{\rm Mir}}}
\newcommand{\Mg}{\mathcal M_g} 
\newcommand{\MCG}{\operatorname{Mod}} 
\newcommand{\Diff}{\operatorname{Diff}} 
\newcommand{\If}{I_f(L,\tau)}
\newcommand{\GOE}{\operatorname{GOE}}
\newcommand{\Ex}{\mathcal E} 
\newcommand{\alow}{\mathbf a}
\newcommand{\adag}{\alow^\dagger} 
\newcommand{\spec}{\mathcal X}
\newcommand{\eigen}{E} 

\title{The Quantum Rabi model: Towards Braak's conjecture}
\date{\today}
\author{Ze\'ev Rudnick}
\address{School of Mathematical Sciences, Tel Aviv University, Tel Aviv 69978, Israel}
\email{rudnick@tauex.tau.ac.il}
 \begin{abstract}
 We establish a density one version of Braak's conjecture on the fine structure of  the spectrum of the quantum Rabi model, as well as a recent conjecture of  Braak, Nguyen, Reyes-Bustos and Wakayama on the nearest neighbor spacings of the spectrum. The proof uses a three-term asymptotic expansion for large eigenvalues due to Boutet de Monvel and Zielinski, and a number theoretic argument from uniform distribution theory.
 \end{abstract}
 \maketitle

\section{Introduction}
In this note we address a conjecture of Braak \cite{Braak} about the fine structure of the spectrum of the quantum Rabi model (QRM),   a fundamental model  of light-matter interaction,  
which describes  the interaction between  a two-level atom (qubit) coupled to a quantized, single-mode harmonic oscillator, see the survey \cite{XZBL}.   

The Hamiltonian of the system is
\[
H_{}= \adag \alow + \Delta \sigma_z  + g\sigma_x (\alow + \adag) 
\]
where $\sigma_x=\left(\begin{smallmatrix}0&1\\1&0 \end{smallmatrix}\right)$, $\sigma_z=\left(\begin{smallmatrix}1&0\\0&-1 \end{smallmatrix}\right)$ are the Pauli matrices of the two-level system, assumed to have level splitting $2\Delta$;  $\adag$ and $\alow$ are the  creation and annihilation operators of  the harmonic oscillator 
with   frequency set to be unity;    and $g>0$ measures the strength of the coupling between the systems.  

The Rabi Hamiltonian  commutes with a parity operator $P=(-1)^{ \adag \alow} \sigma_z$,    and hence the Hilbert space of states decomposes into the $\pm 1$-eigenspaces of $P$ which are preserved by $H_{}$.  
The   Rabi model Hamiltonian in each of the parity eigenspaces can be described 
by the Jacobi matrices 
\[
\begin{pmatrix} d_\pm(0)& a(1)&0& 0& \dots  \\
a(1) & d_\pm(1) & a(2) &0&\dots \\
0& a(2)& d_\pm(2) &a(3) &\dots \\
\dots
\end{pmatrix}
\]
 with
 \[
 d_\pm(k) = k\pm (-1)^k \Delta, \qquad a(k) =  g\sqrt{k} .
 \]

The spectrum of $H_{}$ breaks up into a union of two parity classes. The spectrum in each parity class is non-degenerate, and this allows a unique labeling of the corresponding eigenvalues in increasing order 
\[
\eigen^+_1<\eigen^+_2<\dots <\eigen^+_n<\dots
\]
 and likewise  for the negative parity class $\{\eigen_n^-\}$. 

The  eigenvalues in each parity class satisfy $\eigen_n^\pm =n-g^2+o(1)$ as $n\to \infty$ \cite{Tur, Yanovich}, so that for $n$ sufficiently large, each interval $[n,n+1]$ contains at most $4$ shifted eigenvalues $\eigen_n^\pm+g^2$. 
 Braak \cite{Braak} conjectured that
 \begin{conj}[Braak's G-conjecture]\label{conj:Braak}
  For a given parity class, all intervals $[n,n+1]$ contains at most two shifted eigenvalues, 
  two intervals  containing no shifted eigenvalues are not adjacent, and   two intervals containing two shifted eigenvalues are also not adjacent.  
  \end{conj}

\noindent In this note, we show that Braak's conjecture holds for ``almost all'' $n$, in the following sense:

  
  \begin{thm}\label{main thm}
 Fix $\Delta>0$ and $g>0$. For all but at most $O(N^{3/4+o(1)})$ values of $n\leq N$, the interval $(n,n+1)$ contains exactly two shifted eigenvalues of one of the parity classes,  and none for the other parity class, while the adjacent intervals $(n-1,n)$ and $(n+1,n+2)$ contain exactly two eigenvalues of the other parity class  and none of the first parity class. Moreover, neither $n $ nor $n\pm 1 $ are shifted eigenvalues.  
 
 In particular, almost all intervals $[n,n+1]$ contain exactly two elements of the shifted spectrum.  
    \end{thm}
  
  Concerning the last assertion, there are special choices of the parameters $g$ and $\Delta$ for which there are ``exceptional'' eigenvalues $\eigen$ such that $\eigen+g^2$ is an integer, see \cite[\S 3.2]{XZBL} and the references therein, and our theorem excludes $n-g^2$ being one of these eigenvalues for almost all $n$. 
  
%
%
 %

An application of Theorem~\ref{main thm} is to prove a recent conjecture of Braak, Nguyen, Reyes-Bustos and Wakayama \cite{BNRW} on the nearest neighbor spacings of the full spectrum. 
Denote by $\{\eigen_k\}$   the ordered eigenvalues of $H_{}$ of both parity classes:
\[
\eigen_1\leq \eigen_2 \leq \dots
\]
In \cite{BNRW}, the nearest neighbor spacings $\delta_n:=\eigen_{n+1}-\eigen_n$ were classified into three types: 
{\em positive} if both $\eigen_n,\eigen_{n+1}  $ fell into the positive parity class, {\em negative} if both fell into the negative parity class, and {\em mixed} if one of the pair was positive and one negative.   
Based on numerical observation, it was conjectured \cite[eq 14]{BNRW} that 
 \begin{conj} [Spacings conjecture for the QRM]\label{conj:spacings}
  The frequencies of the three different types of nearest neighbor spacings are $1/4$,$1/4$,$1/2$, respectively. 
  \end{conj}
 This clearly follows from the full conjecture of Braak, but since we establish that Braak's conjecture holds for $100\%$ of $n's$,  we have also established Conjecture~\ref{conj:spacings}.

 Finally, we examine the value distribution of the normalized deviations
\[
\delta_n^\pm:= n^{1/4}\left(\eigen_n^\pm -\left(n-g^2 \right) \right).
\]
As an application of the method of proof of Theorem~\ref{main thm}, we show that the deviations in each parity class satisfy an arcsine law: 
\begin{thm}\label{thm:arcsine law}
For any subinterval $[\alpha,\beta]\subset [ -\frac{\Delta}{\sqrt{2\pi g}}, \frac{\Delta}{\sqrt{2\pi g}}]$, we have
\[
\lim_{N\to \infty} \frac 1N \#\Big\{n\leq N: \delta_n^\pm \in [\alpha,\beta]\Big\} 
= \int_\alpha^\beta \frac {dy}{\pi \sqrt{ \frac{2\pi g}{\Delta^2} -y^2}}  .
\]
\end{thm}

The proof of Theorem~\ref{main thm} starts with an approximation to the eigenvalues due to Boutet de Monvel and Zielinski \cite{BZ}    and concludes with a number-theoretic argument. 

\bigskip

 \noindent{\bf Acknowledgement:} I thank  Daniel Braak, Eduard Ianovich and Masato Wakayama for helpful discussions, and a referee for corrections to an earlier version of the paper. 
 
 This research was supported by the European Research Council (ERC) under the European Union's Horizon 2020 research and innovation programme (grant agreement No. 786758).

\section{The case of good $n$'s}

 Boutet de Monvel and Zielinski \cite{BZ} proved a three term expansion for the eigenvalues in each parity class:
\begin{equation}\label{eq:BZ}
\eigen_n^\pm = n-g^2 \mp \frac{\Delta}{\sqrt{2\pi g}}  \frac{(-1)^n \cos (\theta_n)}{n^{1/4}} + O(n^{-1/2+o(1)}) ,
\end{equation}
where
\[
\theta_n = 4g\sqrt{n}-\frac \pi 4 .
\]
(this approximation was apparently proposed in \cite{FKU}, see also \cite{Yanovich}).

Fix $\delta \in (0,1/4)$ small and let $N\gg 1$. We say that $n\in [N/2,N]$ is ``good'' if  
$$|\cos (\theta_n)| >N^{-1/4+\delta}.$$ 
Otherwise we say that  $n$  is ``bad''. 



Let $x_n^\pm = \eigen_n^\pm+ g^2$ be the shifted eigenvalues,  and denote by  $\spec^\pm  = \{x_n^\pm \}$  the shifted spectra in each parity class.  

\begin{prop}
Let $N\geq N_0(g,\Delta)$ be sufficiently large. 
If  $n\in [N/2,N]$ is ``good'' then  $n$, $n\pm 1$ are not shifted eigenvalues  and either 

i) the interval $(n,n+1)$ contains both $x_n^-$ and $x_{n+1}^-$: $(n,n+1)\cap \spec^- =\{x_n^-,x_{n+1}^-\}$, 
and no elements of $\spec^+$ while the intervals $(n-1,n)$ and $(n+1,n+2)$ contain no elements of $\spec^-$, $(n-1,n)$ contains both $x_{n-1}^+$ and $x_n^+$, while $(n+1,n+2)$ contains both $x_{n+1}^+$ and $x_{n+2}^+$. 

or 

ii) Otherwise, the same holds with the roles of $\spec^-$ and $\spec^+$ reversed.

\end{prop}


\begin{proof}
Let $N\geq N_0(g,\Delta)$  be large, and take $n\in [N/2,N]$. 
Then 
\[
\theta_{n+1} = \theta_n + O\left(\frac 1{\sqrt{N}}\right)
\]
since
\[
\theta_{n+1} - \theta_n   = 4g\sqrt{n+1}-4g\sqrt{n} = \frac{4g}{\sqrt{n+1}+\sqrt{n}} \sim \frac{2g}{\sqrt{N}} .
\]
Hence
\[
\cos(\theta_{n+1}) = \cos(\theta_n) + O\left( \frac 1{\sqrt{N}} \right) ,
\]
and likewise
\[
\cos(\theta_{n-1}) = \cos(\theta_n) + O\left( \frac 1{\sqrt{N}} \right) .
\]
and the same holds for $\cos(\theta_{n\pm 2})$. 

Therefore, for ``good'' $n$,  if $\cos (\theta_n)>N^{-1/4+\delta}$ then $\cos (\theta_{n \pm 1}), \cos (\theta_{n \pm 2})>\frac 12 N^{-1/4+\delta}$ 
and in particular have the same sign as $\cos(\theta_n)$, and and analogous statement holds true if $\cos(\theta_n)<-N^{-1/4+\delta}$. 

Let's assume that $(-1)^n\cos(\theta_n)>N^{-1/4+\delta}$. Then 
\[
x_n^- -n =  \frac{\Delta}{\sqrt{2\pi g}}  \frac{(-1)^n \cos (\theta_n)}{n^{1/4}} + O(n^{-1/2+o(1)})  >\frac 12 \frac{\Delta}{\sqrt{2\pi g}} N^{-1/4+\delta} >  0
\]
so that $x_n^- \in (n,n+1)$. 
Moreover,  
\[
 (-1)^{n+1}\cos(\theta_{n\pm 1})=- (-1)^n\cos(\theta_{n\pm 1})<-\frac 12 N^{-1/4+\delta}<0
 \]
  because $\cos(\theta_{n\pm 1})$ has the same sign and roughly the same size as $\cos(\theta_n)$.  Hence
\[
\begin{split}
x_{n+1}^- -(n+1) &=  \frac{\Delta}{\sqrt{2\pi g}}  \frac{(-1)^{n+1} \cos (\theta_{n+1})}{n^{1/4}} + O(n^{-1/2+o(1)})
\\
& < 
- \;  \frac 14\frac{\Delta}{\sqrt{2\pi g}} N^{-1/2+\delta}<0  
\end{split}
\]
so that $x_{n+1}^-\in (n,n+1)$. Likewise $x_{n-1}^-< n-1$  so that $x_{n-1}^-\in (n-2,n-1)$, and $x_{n+2}^-\in (n+2,n+3)$, $x_{n-2}^-\in (n-2,n-1)$.  
Thus
\[
\spec^- \cap (n,n+1) =\{x_n^-,x_{n+1}^-\},
\]
\[
  \spec^- \cap (n-1,n) = \emptyset=\spec^- \cap (n+1,n+2)
\] 
 in this case.  
Furthermore, for the other parity class, we have
\[
x_n^+ -n = -\; \frac{\Delta}{\sqrt{2\pi g}}  \frac{(-1)^n \cos (\theta_n)}{n^{1/4}} + O(n^{-1/2+o(1)})<- \frac 12 \frac{\Delta}{\sqrt{2\pi g}} N^{-1/2+\delta}<0
\]
so that $x_n^+\in (n-1,n)$, and arguing as above we see that $x_{n+1}^+,x_{n+2}^+\in (n+1,n+2)$ and $x_{n-1}^+,x_{n-2}^+\in (n-1,n)$, so that 
\[
  \spec^+ \cap (n-1,n) = \{x_{n-2}^+, x_{n-1}^+\},
\quad 
\spec^+ \cap (n+1,n+2) = \{x_{n+1}^+, x_{n+2}^+\}
\] 
and $\spec^+\cap (n,n+1) = \emptyset$.

If $(-1)^n\cos(\theta_n)<- N^{-1/4+\delta}$ then we reverse the roles of the parity classes. 
\end{proof}

\section{Bounding the exceptional set}
To conclude the proof of Theorem~\ref{main thm},
we need to bound the number of ``bad'' $n\in [N/2,N]$,   
 that is $|\cos(\theta_n)|<N^{-1/4+\delta}$, which follows from
 \[
 \theta_n \bmod \pi \in \Big[\frac \pi 2 -N^{-1/4+\delta}, \frac \pi 2+N^{-1/4+\delta} \Big]
 \]
 or from
  \[
  (( \frac{ 4g}{\pi} \sqrt{n} + \frac 14 ))   \in \Big[ - \frac{ N^{-1/4+\delta}}{\pi},   \frac{ N^{-1/4+\delta}}{\pi} \Big]
 \]
 where $(( x )) =x-\lfloor x \rfloor\in [0,1) $ denotes the fractional part.
 
 An elementary argument due to Fej\'er (1920)  (see e.g. \cite[Chapter 1 \S 2]{KN}) 
 shows that for any $a>0$, and any shift $\gamma\in \R$, 
 for suitable $c=c(a,\gamma)>0$, for $N\gg 1$, for any interval $[\alpha,\beta]\in [0,1]$, 
\begin{equation*}
\left|  \#\left\{n\in [N/2,N]: ((a\sqrt{n} +\gamma )) \in [\alpha, \beta] \right\} - (\beta-\alpha)\frac N2 \right|\leq c\sqrt{N}.
\end{equation*} 
and in particular, for an interval of length $ >N^{-1/2+o(1)}$ 
   the number of fractional parts which fall into that interval is asymptotically $N/2$ times the length of that interval.  In our case, the length of the interval is $\frac{2}{\pi}N^{-1/4+\delta}$ and we obtain that the number of ``bad'' $n\in [N/2,N]$ is about $N^{3/4+\delta}$, as claimed.

For the readers' benefit, we recall Fej\'er's argument for the case of  the fractional parts of $\sqrt{n}$.  If $k^2\leq n<(k+1)^2$ then 
$((\sqrt{n})) = \sqrt{n}-k$,  
and then $((\sqrt{n}))\in [\alpha,\beta]$ means $\alpha\leq \sqrt{n}-k \leq \beta$ or $(k+\alpha)^2\leq n \leq (k+\beta)^2$, so that $n$ lies in an interval of length $2k(\beta-\alpha)+O(1)$. Summing over $k$ we see that the number of $n\in [N/2,N]$ with $((\sqrt{n}))\in [\alpha,\beta]$ is 
\[
\sum_{\sqrt{N/2}\leq k <\sqrt{N}} \left\{ 2k\left(\beta-\alpha\right)+O\left(1\right) \right\} =  (\beta-\alpha) \frac{N}{2} + O(\sqrt{N}).
\]

\section{Proof of Theorem~\ref{thm:arcsine law}}

\begin{proof}
We want to count $ \#\Big\{n\leq N: \delta_n^- \in [\alpha,\beta]\Big\} $. 
According to \eqref{eq:BZ}, we have
\[
\delta_n^- =C (-1)^n \cos\left( 4g\sqrt{n}-\frac \pi 4\right)  + O(n^{-1/4+o(1)})
\]
with
\[
C=  \frac{\Delta}{\sqrt{2\pi g}} .
\]
For the purpose of understanding the distribution of $\delta_n^-$, we may ignore the remainder term.

Writing   $\varphi_n = \frac{2g}{\pi }\sqrt{n} - \frac 18$, 
we observe that $\cos\left( 4g\sqrt{n}-\frac \pi 4\right)  = \cos(2\pi \varphi_{n})$ depends only on the fractional part $((\varphi_n))\in [0,1)$ of $\varphi_n$. 

We split the range $n\in [1,N]$ into even and odd $n$'s. 
First take $n=2m$ even.  Then  $\delta_{2m}^- \in [\alpha,\beta]$ is equivalent to 
\[
\widehat\alpha \leq  \cos (2\pi \varphi_{2m} ) \leq \widehat\beta
\]
where 
\[
\widehat\alpha:=\frac{\alpha}{C}, \quad \widehat\beta:= \frac{\beta}{C}.
\]
Writing $\mathbf 1_{[\widehat \alpha,\widehat \beta]}$ for the indicator function  of the interval $[\widehat \alpha,\widehat \beta]$, we have
\[
  \#\Big\{ 2m\leq N: \delta_{2m}^-\in [\alpha,\beta] \Big\} = 
 \sum_{m=1}^{N/2} \mathbf 1_{[\widehat \alpha,\widehat \beta]}\left(\cos\left( 2\pi \varphi_{2m}\right)\right) .
\]
Now the function $\mapsto  \mathbf 1_{[\widehat \alpha,\widehat \beta]}\left(\cos \left(2\pi x\right)\right)$ is Riemann integrable, and from uniform distribution of the fractional parts of $\varphi_{2m}$ (Fej\'er's theorem) it follows that (see e.g. \cite[Chapter 1, \S1]{KN})
\[
\lim_{N\to \infty} \frac 1{  N/2  } \sum_{m=1}^{N/2} \mathbf 1_{[\widehat \alpha,\widehat \beta]}\left(\cos\left( 2\pi \varphi_{2m}\right)\right) =\int_{-1/2}^{1/2}  \mathbf 1_{[\widehat \alpha,\widehat \beta]}\left(\cos \left(2\pi x\right)\right) dx  . 
\]
Now
\[
\begin{split}
\int_{-1/2}^{1/2}  \mathbf 1_{[\widehat \alpha,\widehat \beta]}\left(\cos \left(2\pi x\right)\right) dx  
&= \frac 1\pi \int_0^\pi  \mathbf 1_{[\widehat \alpha,\widehat \beta]}\left(\cos \left(y\right)\right) dy 
\\
& =
   \frac 1{\pi} \int_{\widehat \alpha}^{\widehat\beta} \frac{dt}{\sqrt{1-t^2}}
   \\
   &= \frac 1{\pi} \int_\alpha^\beta \frac{C dy}{\sqrt{1-C^2 y^2}}  
   =  \int_\alpha^\beta \frac {dy}{\pi \sqrt{ \frac{2\pi g}{\Delta^2} -y^2}} .
\end{split}
\]
Therefore we obtain
\[
\lim_{N\to \infty} 
\frac 1{ N   } \#\Big\{ 2m\leq N: \delta_{2m}^-\in [\alpha,\beta] \Big\} = \frac 12 \int_\alpha^\beta \frac {dy}{\pi \sqrt{ \frac{2\pi g}{\Delta^2} -y^2}} .
\]

The same considerations are valued in the case that $n=2m+1$ is odd, except that we require
\[
 \widehat\alpha \leq  - \cos (2\pi \varphi_{2m} ) \leq \widehat\beta ,
\]
leading to   the integral 
\[
\int_{-1/2}^{1/2}  \mathbf 1_{[-\widehat \beta,-\widehat \alpha]}\left( \cos \left(2\pi x\right)\right) dx  
\]
which gives the same result.

Altogether we obtain
\[
 \lim_{N\to \infty}\frac 1N \#\Big\{\delta_n^- \in [\alpha,\beta] \Big\} 
 = \int_\alpha^\beta \frac {dy}{\pi \sqrt{ \frac{2\pi g}{\Delta^2} -y^2}} .
\]  
The argument for $\delta_n^+$ is identical. 
\end{proof}

\end{document}